\documentclass[a4paper,11pt,english]{amsart}
\usepackage{amssymb,babel,xspace}
\pdfoutput=1
\usepackage{hyperref}

\title[CR quadrics with a symmetry property]{A characterization of CR quadrics\\ with a symmetry property}

\author[A. Altomani]{Andrea Altomani}
\address{A.\ Altomani:
Research Unit in Mathematics \\
University of Luxembourg\\
6, rue Coudenhove-Kalergi\\
L-1359 Luxembourg}
\email{andrea.altomani@uni.lu}

\author[C. Medori]{Costantino Medori}
\address{C.\ Medori:
Dipartimento di Matematica\\
Universit\`a di Parma\\ Parco Area
del\-le Scien\-ze 53/A \\
43124 Parma (Italy)}
\email{costantino.medori@unipr.it}

\date{\today}



\numberwithin{equation}{section}

\theoremstyle{plain}
\newtheorem{thm}{Theorem}[section]
\newtheorem{lem}[thm]{Lemma}
\newtheorem{cor}[thm]{Corollary}
\newtheorem{prop}[thm]{Proposition}

\theoremstyle{definition}
\newtheorem{exam}[thm]{Example}
\newtheorem{defn}[thm]{Definition}

\newtheorem{remark}[thm]{Remark}

\newcommand{\CR}{\texorpdfstring{\ensuremath{\mathrm{CR}}}{CR}\xspace}
\newcommand{\C}{\mathbb{C}}
\newcommand{\N}{\mathbb{N}}
\newcommand{\R}{\mathbb{R}}
\newcommand{\Z}{\mathbb{Z}}
\newcommand{\e}{\mathrm{e}}
\newcommand{\gr}{\mathbf}
\newcommand{\G}{\gr{G}}
\renewcommand{\P}{\gr{P}}
\newcommand{\la}{\mathfrak}
\newcommand{\ga}{\la{a}}
\newcommand{\g}{\la{g}}
\newcommand{\p}{\la{p}}
\newcommand{\q}{\la{q}}

\newcommand{\h}{\la{h}}

\newcommand{\gn}{\la{n}}
\newcommand{\gs}{\la{s}}
\newcommand{\rad}{\la{r}}

\newcommand{\ii}{\mathrm{i}}
\newcommand{\Id}{\mathrm{Id}}

\newcommand{\B}{\mathcal{B}}

\DeclareMathOperator{\Ad}{Ad}
\DeclareMathOperator{\ad}{ad}
\DeclareMathOperator{\Aut}{Aut}
\DeclareMathOperator{\Int}{Int}

\begin{document}
\begin{abstract}
We study CR quadrics satisfying a symmetry property $(\tilde S)$ which is slightly weaker than the symmetry  property $(S)$,
recently introduced by W. Kaup, which requires the existence of an automorphism reversing the gradation of the Lie algebra of infinitesimal automorphisms of the quadric.

We characterize quadrics satisfying the $(\tilde S)$ property in terms of their Levi-Tanaka algebras. 
In many cases the $(\tilde S)$ property implies the $(S)$ property; this holds in particular for compact quadrics.

We also give a new example of a quadric such that the dimension of the 
algebra of positive-degree infinitesimal automorphisms is larger than the dimension of the quadric.
\end{abstract}
\maketitle
\section{Introduction}
The affine quadrics
$Q=\{(z,w)\in\mathbb{C}^n_z\times\mathbb{C}^k_w\mid \mathrm{Im}\,w=H(z,z)\}$,
for a nondegenerate hermitian form $H\colon\C^n\times\C^n\to\C^k$,
provide the simplest nontrivial
examples of \CR manifolds.
Any such quadric $Q$ has a canonical completion 
as a real submanifold $\hat{Q}$, not necessarily closed when $k>1$, of a complex
projective space $\mathbb{CP}^{N}$. Both $Q$ and its completion $\hat{Q}$ are
\CR-homogeneous and any local \CR automorphism of $Q$ or $\hat{Q}$ extends to
a global projective automorphism of $\mathbb{CP}^{N}$ transforming $\hat{Q}$
into itself
(see \cite{IsKa:2010}).
\par
The choice of a point of ${Q}$, e.g. the point $0\in\mathbb{C}^n\times\mathbb{C}^k$, yields a natural gradation
\[
  \mathfrak{g}=\mathfrak{g}_{-2}\oplus\mathfrak{g}_{-1}\oplus\mathfrak{g}_{0}
\oplus\mathfrak{g}_{1}\oplus\mathfrak{g}_{2}
\]
of the Lie algebra $\mathfrak{g}$ of the infinitesimal \CR automorphisms of
${Q}$.
The group $\mathrm{Aut}_{CR}(\hat{Q})$ of \CR automorphisms of $\hat{Q}$ has $\g$ as its Lie algebra and
acts on $\g$ via the adjoint action.
\par

In \cite{Kaup:2010} W. Kaup
defined a quadric $Q$ to have the \textit{symmetry property}
$(S)$  if there is an involutive \CR automorphism $\gamma$ of $\hat{Q}$
such that $\mathrm{Ad}(\gamma)(\mathfrak{g}_j)=\mathfrak{g}_{-j}$ for $j=0,\pm{1},\pm{2}$.
\par

In this paper we
generalize property $(S)$ to an $(\tilde{S})$, that requires the
existence of a degree reversing \CR automorphism $\gamma$ of $\hat{Q}$ of finite
order. 
\par
Quadrics enjoying property $(\tilde{S})$ are characterized in terms of
a property of the Levi-Tanaka algebra $\mathfrak{g}$. We prove that {$(\tilde{S})$
holds true if and only if the gradation of $\mathfrak{g}$ is inner and
defined by a semisimple element of a Levi factor of $\mathfrak{g}$.}
The order of $\gamma$ can be required to be either $2$, or $4$, and we show that
in several instances, including the case where $\hat{Q}$ is compact, 
we get actually property $(S)$.  
We have no example
to show that $(\tilde{S})$ and $(S)$  are not equivalent.

We also provide a simple example of a \CR quadric $Q$ for which $\dim\g_1>\dim\g_{-1}$ and  $\dim\g_2>\dim\g_{-2}$, giving a new
counterexample to a question that
was formulated by V.Ezhov and G.Schmalz in \cite{EzSc:2001}.

The authors wish to thank I. Kossovskiy for pointing out the result of Utkin \cite{Ut:2002}.

\section{\CR manifolds and Levi-Tanaka algebras}
We first recall the definition of a \CR manifold. 

\begin{defn}
A \emph{\CR manifold} of type $(n,k)$ is the datum $(M,T^{1,0}M)$ of a real smooth manifold $M$ of dimension $2n+k$ and a smooth complex vector subbundle $T^{1,0}M$, with constant complex rank $n$, of the complexification $T^\C M$ of the tangent bundle of $M$, satisfying the following conditions:
\begin{enumerate}
\item $T^{1,0}M\cap\overline T^{1,0}M=0$,
\item $[C^\infty(M,T^{1,0}M),C^\infty(M,T^{1,0}M)]\subset C^\infty(M,T^{1,0}M)$.
\end{enumerate}
The integers $n$ and $k$ are called the \emph{\CR dimension} and \emph{\CR codimension} of $M$. We also set $T^{0,1}M=\overline{T^{1,0}M}$ and $HM=TM\cap (T^{1,0}M+T^{0,1}M)$.

Let $J\colon T^{1,0}M+T^{0,1}M\to T^{1,0}M+T^{0,1}M$ be the linear semisimple isomorphism with eigenvalues $\ii$ on $T^{1,0}M$ and $-\ii$ on $T^{0,1}M$. Then $J$ preserves $HM$ and is called a \emph{partial complex structure}.

A \emph{\CR map} between two \CR manifolds $M$ and $N$ is a smooth map $f\colon M\rightarrow N$ such that $df^\C(T^{1,0}M)\subset T^{1,0}N$. The notions of \CR isomorphism and automorphism are defined in the natural way.
\end{defn}

\begin{defn}
Let $M$ be a real submanifold of a complex manifold $X$. Define $T^{1,0}M=T^\C M\cap T^{1,0}X$. If $T^{1,0}M$ has constant rank, then $(M,T^{1,0}M)$ is a \CR manifold, called a \emph{\CR submanifold} of $X$.
\end{defn}

We introduce two further definitions.

\begin{defn}
A CR manifold $(M,T^{1,0}M)$ is said to be:
\begin{enumerate}
\item \emph{Levi nondegenerate}  at a point $x\in M$ if for every vector field $Z\in C^\infty(M,T^{1,0}M)$ there exists a vector field $\bar W\in C^\infty(M,T^{0,1}M)$ such that $[Z,\bar W]_x\notin T^{1,0}_xM+T^{0,1}_xM$;
\item of \emph{finite type} at a point $x\in M$ if the Lie algebra generated by all vector fields in $C^\infty(M,T^{1,0}M+T^{0,1}M)$ spans $T^\C_xM$.
\end{enumerate}
\end{defn}

\subsection{Levi-Tanaka algebras and standard \CR manifolds}

Let $M$ be a \CR manifold, and $x\in M$ a point where $M$ is Levi nondegenerate and of finite type. We associate to $x$ a graded Lie algebra, called the Levi-Tanaka algebra of $M$ at $x$. We refer to~\cite{MN:1997} for a more detailed discussion of Levi-Tanaka algebras.
Define:
\[ \mathcal{D}_{0}=0,\qquad \mathcal{D}_{-1}=C^\infty(M,HM), \]
and inductively, for $p\geq 2$:
\[ \mathcal D_{-p}=\mathcal D_{-p+1}+[\mathcal D_{-p+1},\mathcal D_{-1}]. \]
Then we set, for $p\geq 1$:
\[\la m_{-p}=\mathcal D_{-p}(x)/\mathcal D_{-p+1}(x).\]
The vector field bracket induces a graded Lie algebra structure on $\la m_-=\sum_{p\leq -1} \la m_p$. Note that $\la m_{-1}$ is canonically isomorphic to $H_xM$.
Then it is naturally defined a complex structure $J$ on $\la m_{-1}$ and $[JX,JY]=[X,Y]$ for every $X,Y\in \la m_{-1}$.

Let
\[ \la m_0=\{ D\in\mathrm{Der}_0(\la m_-)\mid [D|_{\la m_{-1}},J]=0\} \]
be the set of zero-degree derivations on $\la m_-$ commuting with $J$ on $\la m_{-1}$. Then $\la m_0+\la m_-$ is a graded Lie algebra.

\begin{defn}\label{defn:LT}
The \emph{Levi-Tanaka algebra} associated to $M$ at a point $x\in M$ where $M$ is Levi nondegenerate and of finite type is the (unique) graded Lie algebra $\g=\sum_{p\in\Z}\g_p$ with the following properties:
\begin{enumerate}
\item $\g_p=\la m_p$ for $p\leq 0$,
\item for all $X\in\g_p$, with $p\geq 0$, the action $\ad_\g(X)|_{\g_{-1}}$ is nonzero,
\item $\g$ is maximal with those properties.
\end{enumerate}
\end{defn}

\begin{defn}
In general, we can start with any graded Lie algebra $\la m_-=\sum_{p\leq -1} \la m_p$, such that $\la m_{-1}$ generates $\la m_-$ and with a complex structure $J$ on $\la m_{-1}$ such that
\[ [X,Y]=[JX,JY] \quad \forall\,  X,Y\in\la m_{-1},\]
and perform the same prolongation procedure as in Definition~\ref{defn:LT}.

The resulting algebra $\g=\sum_{p\in\Z}\g_p$ (with complex structure $J$ on $\g_{-1}$) is a \emph{Levi-Tanaka algebra}.
\end{defn}

We fix the following notation: for a graded Lie algebra $\g=\sum_{p\in\Z}\g_p$, we set 
\[
\begin{aligned}
\g_-&=\sum_{p<0}\g_p, &   \p&=\sum_{p\geq0}\g_p, \\[3pt] 
\g_+&=\sum_{p>0}\g_p, &  \p^\mathrm{opp}&=\sum_{p\leq0}\g_p.
\end{aligned}
\]

A Levi-Tanaka algebra has trivial center and contains a unique element $E\in\g_0$, called \emph{characteristic element}, such that $\ad_{\g}(E)|_{\g_j}=j\,\Id_{\g_j}$ for all $j\in\Z$.

\begin{defn}
For a Levi-Tanaka algebra $\g=\sum_{p\in\Z}\g_p$, with complex structure $J_{\g}$ on $\g_{-1}$, it is possible to construct a \CR manifold such that the associated Levi-Tanaka algebra at every point is isomorphic to $\g$. 

Let $\tilde\G$ be the connected and simply connected group with Lie algebra $\g$, and $\P$ the analytic subgroup with Lie algebra $\p$. Then $\P$ is closed, and we let $S=S(\g)=\tilde\G/\P$. 

There is a unique $\tilde\G$-homogeneous \CR structure on $S$ such that at the base point $o=e\P$ the partial complex structure is given by $H_{o}S=\g_{-1}$ and $J_o=J_\g$, where we identified $T_oS$ and $\g_-$, in the natural way. 

The \CR manifold $S=S(\g)$ is the \emph{standard \CR manifold} associated to $\g$ (see \cite{MN:1997}).
\end{defn}

The standard \CR manifold $S=S(\g)$ is simply connected, and  $\g$ is isomorphic both to the Lie algebra of its infinitesimal automorphisms and to the Lie algebra of the group of (global) \CR automorphisms.
 
We recall that $S$ is compact if and only if $\g$ is semisimple (see \cite[Corollary 5.3]{MN:2000}).
The group of \CR automorphisms of $S$ is in general not connected, but we can give a description of its connected component of the identity.
\begin{prop}
Let $\g$ be a Levi-Tanaka algebra, and $S$ the associated standard \CR manifold. Then the connected component of the identity of the group $\mathrm{Aut}_{\CR}(S)$ of \CR automorphisms of $S$ is the group $\G^0=\mathrm{Int}(\g)$ of inner automorphisms of $\g$.
\end{prop}
\begin{proof}
Let $\P^0=\{g\in\G^0\mid \Ad(g)(\p)=\p\}$. Then the Lie algebra of $\P^0$ is $\p$ and the manifold $M=\G^0/\P^0$ has a natural \CR structure. The natural quotient $\tilde\G\rightarrow\G^0$ induces a covering map $\pi\colon S\rightarrow M$.

The manifold $M$ is maximally homogeneous 
(the dimension of its automorphism group is equal to the dimension of the group of automorphisms of the standard manifold $S$), 
then $M$ is isomorphic to $S$ and, in particular,
simply connected (see~\cite{MN:2001}). It follows that $\pi$ is a diffeomorphism.
\end{proof}
The standard \CR manifold $S$ can then be identified to the set of inner conjugates of $\p$ in $\g$.

This observation provides also another construction of standard \CR manifolds. The group $\G^0$ acts, via the complexification of the adjoint action, on all the complex grassmannians of subspaces of $\g^{\C}$. Let 
\[
\la{q}=\g_2^{\C}+\g_1^{\C}+\g_0^{\C}+\{X+\ii JX\mid X\in\g_{-1}\}.
\]
The $\G^0$-orbit through the point $o=\la{q}$ in the complex grassmannian $\mathrm{Gr}_{\dim\la{q}}(\g^{\C})$, with the \CR structure given by the embedding, is \CR-isomorphic to the standard \CR manifold associated to $\g$.

Although we will not use it, we give a characterization of the full automorphism group of a standard \CR manifold.
\begin{prop}
Let $\g$ be a Levi-Tanaka algebra, and $S$ the associated standard \CR manifold. Then the group $\mathrm{Aut}_{\CR}(S)$ of \CR automorphisms of $S$ is the group 
\[
\G=\{g\in\Aut(\g)\mid \text{$g\cdot \la q$ is $\mathrm{Int}(\g)$-conjugate to $\la q$}\}.
\]
\end{prop}
\begin{proof}
$(\mathrm{Aut}_{\CR}(S) \subset \G)$.
Let $\phi\in\mathrm{Aut}_{\CR}(S)$ be a \CR automorphism of $S$ and $X\in\g$. Denote by $X^\dagger$ the vector field on $S$ generated by $X$. Then $(\phi\cdot X)^\dagger=d\phi(X^\dagger)$ defines an action of $\phi$ on $\g$, which is an automorphism. 
Let $g\in\mathrm{Int}(\g)$ be an element such that $\phi\circ g(o)=o$. Then $\phi\circ g\cdot\p=\p$, and $\phi\circ g\cdot\la q=\la q$ because $\phi\circ g$ is a \CR map.

$(\mathrm{Aut}_{\CR}(S) \supset \G)$. Let $g$ be an element of $\G$, and $h\in\mathrm{Int}(\g)$ an element with $g\cdot\la q=h\cdot\la q$. Define an action of $g$ on $S$ as follows: for $k\in\mathrm{Int}(\g)$, let $g\cdot(k\cdot o)=(gkhg^{-1})\cdot o$.
\end{proof}

\section{CR quadrics}
Let $H\colon\C^n\times\C^n\to \C^k$ be a vector valued hermitian form, linear in the first variable and 
anti-$\C$-linear in the second one.
\begin{defn}
The vector valued hermitian form $H\colon\C^n\times\C^n\to \C^k$ is said to be:
\begin{description}
\item[nondegenerate] if for all $z\in\C^n\setminus\{0\}$ there exists $z'\in\C^n$ such that $H(z,z')\neq 0$;
\item[fundamental] if the set $\{H(z,z)\mid z\in\C^n\}\subset\R^k$ spans $\R^k$.
\end{description}
\end{defn}
To a vector valued hermitian form it is naturally associated a \CR submanifold of $\C^{n+k}$ in the following way.
\begin{defn}
The \emph{affine \CR quadric associated to} a vector valued hermitian form $H\colon\C^n\times\C^n\to \C^k$ is the \CR-submanifold of $\C^n\oplus\C^k$ given by:
\[
Q=Q^H=\{(z,w)\in\C^n\oplus\C^k\mid \Im{w}=H(z,z)\}.
\]
\end{defn}
It is straightforward to see that $Q^H$ is a \CR manifold of \CR-dimension $n$ and \CR-codimension $k$, it is finitely nondegenerate (in fact Levi nondegenerate) if and only if $H$ is nondegenerate, and it is of finite type (indeed of type $2$) if and only if $H$ is fundamental.
\medskip
\begin{remark}
Any affine quadric $Q$ can be written as a product $Q=Q'\times\C^m\times\R^h$, where $Q'$ is a Levi nondegenerate affine quadric of finite type, $m$ is the dimension of the null space of $H$, and $h$ is the codimension of the image of $H$ in $\C^k$. 
\end{remark}
\medskip

\emph{We assume, from now on, that $H$ is nondegenerate and fundamental.}
\medskip

The Lie algebra of infinitesimal automorphisms of $Q$ is finite-dimen\-sional and
possesses a natural grading $\g=\bigoplus_{i=-2}^{2} \g_i$. 
It is canonically isomorphic to the Levi-Tanaka algebra 
associated to $Q$ (see \cite{T:1967} and \cite{EzIs:1999}). 

Let $(\g=\oplus_{i\in\Z}\g_i, J)$ be the Levi-Tanaka algebra associated to $Q$. Then $\g_i=0$ for $i<-2$ and $i>2$, and $\dim_{\R}\g_{-1}=2n$, $\dim_{\R}\g_{-2}=k$. The Lie algebra structure on $\g_-=\g_{-1}\oplus\g_{-2}$ is given by $H$ in the following way. Identify $\g_{-1}$, 
endowed with its complex structure $J$, to $\C^n$, and $\g_{-2}$ to $\R^k\subset\C^k$.
Then:
\[
[X,Y]=\Im( H(X,Y)),\qquad\forall X,Y\in\g_{-1}.
\]

\begin{defn}
The \emph{quadric} $\hat Q=\hat Q^H$ associated to $H$ is the standard \CR manifold $S(\g)$ associated to $\g$. This definition agrees with the definition in~\cite{Kaup:2010} (see also \cite{IsKa:2010}).
\end{defn}
The affine quadric $Q$ is \CR diffeomorphic to the $\G_-$ orbit through $o$, and it is open and dense in $\hat Q$. The complement $\hat Q\setminus Q$ is the intersection of $\hat Q$ and a complex-algebraic subvariety of $\mathrm{Gr}_{\dim\la{q}}(\g^{\C})$.

In \cite{Kaup:2010}
W. Kaup introduced a symmetry property, called property $(S)$, for the quadric $\hat Q$. Here we consider  the following generalization.
\begin{defn}
The quadric $\hat Q$ is said to have:
\begin{itemize}
\item \emph{the $(S)$ property} if there exists an involutive automorphism $\gamma\in\G$ such that $\Ad_{\g}(\gamma)(E)=-E$;
\item \emph{the $(\tilde S)$ property} if there exists an automorphism $\gamma\in\G$ of finite order such that $\Ad_{\g}(\gamma)(E)=-E$.
\end{itemize}
We use the same notation for the Levi-Tanaka algebra associated to the quadric.
\end{defn}

Our aim is to characterize quadrics with the $(\tilde S)$ property, and show that in many cases the $(S)$ and $(\tilde S)$ properties are equivalent.

\section{Levi-Malcev decomposition}\label{s:Levi}
We recall that Levi Tanaka algebras have a \emph{pseudocomplex graded} Levi-Malcev decomposition, i.e. compatible with the grading and the complex structure~\cite{MN:2002}.
More precisely, given a Levi-Tanaka algebras $\g$, with radical $\rad$, there exist a semisimple subalgebra $\gs$ such that:
\begin{enumerate}
\item $\g=\gs\oplus\rad$;
\item $\gs$ and $\rad$ are graded;
\item $\gs_{-1}$ and $\rad_{-1}$ are $J$-invariant.
\end{enumerate}
\begin{lem}\label{LeviMalcev}
Let $\g=\sum_{p\in\Z} \g_p$ be a finite dimensional Levi-Tanaka algebra, 
and $\rad=\sum\rad_p$ the radical of $\g$. 
If there is an automorphism 
$\Gamma\in\mathrm{Aut}(\g)$ 
with $\Gamma(E)=-E$,
then 
$\rad_{-2}:=\rad\cap\g_{-2}\neq\g_{-2}$.
In particular, $\g$ is not solvable.
\end{lem}
\begin{proof}
Assume $\rad\cap\g_{-2}=\g_{-2}$. Then we have $\rad\cap\g_{p}=\g_{p}$ for every $p\leq -2$.
Let $\gn=\sum\gn_p$ be the nilradical of $\g$ and consider the descending
 central sequence
\[
\gn^1=\gn, \qquad \gn^{k+1}=[\gn^k,\gn], \text{ for } k\geq 1.
\]
Note that $\sum_{p\neq 0} \rad_p \subset\gn$.
Let $d$ be the minimal
 integer such that $\gn^d \neq{0}$ and $\gn^{d+1}=0$.
Then $\gn^d$
is a characteristic ideal of $\g$ and $[\gn,\gn^d]=0$.

Consider
$X\in\gn^d_1:=\gn^d\cap\g_1$.
We have
\begin{equation}
[X,\g_{p}]=[X,\rad_{p}]=[X,\gn_{p}]={0}\,, \quad\forall p\leq -2,
\end{equation}
hence $X=0$ (see \cite[Theorem 3.1]{MN:1997}).
Then $\gn^d_1={0}$ and in general $\gn^d_p:=\gn^d\cap\g_p={0}$, for
any $p>0$.

Since $\Gamma$ interchanges $\g_p$ and $\g_{-p}$, and the ideal $\gn^d$ is characteristic, we have
also $\gn^d_p=0$
for $p\neq 0$, therefore $\gn^d\subset\g_0$.

Finally,
\[
[\gn^d_0,\g_{-1}]\subset \gn^d_{-1}=\{0\},
\]
hence $\gn^d_0=\{0\}$.
Then we have $\gn^d=\{0\}$, obtaining a contradiction.
\end{proof}

We fix now a pseudocomplex graded Levi-Malcev decomposition
\[
\g=\la s\oplus\la r
\]
of the Levi-Tanaka algebra $\g$ associated to a quadric $\hat Q$
having the $(\tilde{S})$ property. 
From Lemma \ref{LeviMalcev} it follows that $\gs_{2}\neq0$ and $\gs_{-2}\neq 0$.

Let $E\in\g$ be the characteristic element. Then $E=E_\gs+E_\rad$ with $E_\rad\in\rad$, 
and $E_\gs\in\gs$ is the characteristic element of $\gs$.

\begin{prop}\label{Eins}
If a quadric $\hat Q$ admits an automorphism $\gamma$ with $\Ad(\gamma)(E)=-E$, then $E=E_\gs$.
\end{prop}
\begin{proof}
Assume that $\hat Q$ admits such a $\gamma$. The isotropy Lie algebra at the point $\gamma\cdot o\in\hat Q$ is $\g_0\oplus\g_-$. It follows that there exists an element $X_+\in\g_+$ with $\exp(X_+)\gamma \cdot o\in Q$. Since $\exp(\g_-)$ acts transitively on $Q$, we also have an element $X_-\in\g_-$ such that $\exp(X_+)\gamma \cdot o=\exp(X_-) \cdot o$ or in other words:
$\gamma=\exp(-X_+)\exp(X_-)h$, where $h$ is an element of the isotropy at $o$. Since the isotropy at $o$ is exactly $\G_0\G_+$, we finally obtain, for a $g_0\in\G_0$ and an $X'_+\in\G_+$:
\begin{align*}
\gamma&=\exp(-X_+)\exp(X_-)\exp(X'_+)g_0\\
  &=\exp(Y^1_1)\exp(Y^1_2)\exp(Y^2_{-1})\exp(Y^2_{-2})\exp(Y^3_1)\exp(Y^3_2)g_0
\end{align*}
(here the subscripts indicate the degrees of the homogeneous elements $Y^i_j$).

From $\Ad(\gamma)(E)=-E$ we obtain
\[
2E=2[Y^2_{-2},Y^3_{2}]+\frac12[Y^2_{-1},Y^1_{1}+Y^3_{1}].
\]
Let $\la n$ be the nilradical of $\g$. Note that it is graded, and $\rad_p=\la n_p$ for all $p\neq 0$. Decompose each element $Y^i_j$ into its $\gs$ and $\gn$ component. It follows
\begin{align*}
2E_{\rad}\in([\gs,\la n]+[\rad,\rad])\cap\g_0\subset \la n_0
\end{align*}
and $E_\rad$ is $\ad$-nilpotent.

Since $\ad(E)$ preserves $\gs$, and $\rad$ is an ideal, we have $\ad(E)|_\gs=\ad(E_\gs)|_\gs$, and $E_\gs$ is a $\ad$-semisimple element of $\gs$. Then $E=E_\gs+E_\rad$ is a Wedderburn decomposition of $E$, and since $E$ is semisimple element, it follows that $E=E_{\gs}$.
\end{proof}
We also have the following
\begin{lem}\label{lem:p-opp-conj}
If the quadric $\hat Q$ has property $(\tilde S)$, then $\p^{\mathrm{opp}}$ is conjugate to $\p$ by an inner automorphism of $\g$.
\end{lem}
\begin{proof}
Assume that $\hat Q$ has property $(\tilde S)$. Since $\Int(\g)$ acts transitively on $\hat Q$ and $\Ad(\gamma)(\p)=\p^{\mathrm{opp}}$ is the isotropy Lie algebra at the point $\gamma\cdot o\in\hat Q$, the condition is necessary.
\end{proof}

\section{The semisimple case}
We assume now that $\g$ is semisimple. For a standard \CR manifold (and in particular for quadrics $\hat{Q}^H$) this is equivalent to compactness
(see \cite[Corollary 5.3]{MN:2000}).
First we recall the description of semisimple Levi-Tanaka algebras (see~\cite{MN:1998} for a more detailed treatment of the topic). 

Let $\g$ be a semisimple Levi-Tanaka algebra.
Since every semisimple Levi-Tanaka algebra is a direct sum of simple Levi-Tanaka algebras, we can assume that $\g$ is simple. Choose a maximally noncompact Cartan subalgebra $\la h$ of $\g$, let $\g^\C$ and $\la h^\C$ be the complexifications of $\g$ and $\la h$, and $\mathcal R=\mathcal R(\g^\C,\la h^\C)$ be the corresponding root system.

The conjugation of $\g^\C$ with respect to the real form $\g$ leaves $\la h^\C$ invariant, and then induces a conjugation $\sigma\colon\mathcal R\rightarrow\mathcal R$. Let $\mathcal R^\bullet=\{\alpha\in\mathcal R\mid\sigma\alpha=-\alpha\}$ be the set of compact roots. There exists a choice of a set of positive  roots $\mathcal R_+$ such that $\sigma(\mathcal R_+\setminus\mathcal R^\bullet)\subset\mathcal R_+$. Let $\mathcal B$ be the system of simple positive roots for $\mathcal R_+$, that we identify to the nodes of the associated Dynkin diagram $\Delta$, and $\mathcal B^\bullet=\mathcal B\cap\mathcal R^\bullet$.

The action of the conjugation $\sigma$ on simple positive roots can be described as follows: there exists an involution of the Dynkin diagram $\epsilon\colon\mathcal B\to\mathcal B$ such that $\sigma\alpha-\epsilon\alpha\in\langle\mathcal B^\bullet\rangle_\Z$. The datum of $(\Delta,\mathcal B^\bullet, \epsilon)$ completely determines $\g$ and is known as Satake diagram of $\g$.

Fix a subset $\Phi$ of $\mathcal B$ with the following properties:
\begin{enumerate}
\item $\Phi\cap\B^\bullet=\emptyset$,
\item $\Phi\cap\epsilon\Phi=\emptyset$ (in particular $\epsilon$ is nontrivial),
\item  every connected component of $\Delta$ intersects both $\Phi$ and $\epsilon\Phi$,
\item every path in $\Delta$ connecting two elements of $\Phi$ contains elements of $\epsilon\Phi$.
\end{enumerate}
Let $E,J$ be the elements of $\la h$ such that: 
\begin{align*}
&\begin{cases}
\alpha(E)=1&\text{for $\alpha\in\Phi\cup\epsilon(\Phi)$,}\\
\alpha(E)=0&\text{for $\alpha\notin\Phi\cup\epsilon(\Phi)$,}
\end{cases}\\
&\begin{cases}
\alpha(J)=-\ii&\text{for $\alpha\in\Phi$,}\\
\alpha(J)=\ii&\text{for $\alpha\in\epsilon(\Phi)$,}\\
\alpha(J)=0&\text{for $\alpha\notin\Phi\cup\epsilon(\Phi)$.}
\end{cases}
\end{align*}

Then $E$ defines a gradation on $\g$, and $J$ defines a complex structure on $\g_{-1}$. The largest $p\in\N$ such that $\g_p\neq \{0\}$ is called the \emph{kind} of $\g$. It coincides with the degree of a maximal positive root.
Conversely, every simple Levi-Tanaka algebra is isomorphic to one obtained in this way.

It is straightforward then to classify the simple Levi-Tanaka algebras of kind $2$. The names of the simple Lie algebras of real type are those of the corresponding symmetric spaces in Cartan's classification, the order of the roots $\{\alpha_1,\ldots,\alpha_\ell\}=\mathcal B$ follows Bourbaki (see the table in the appendix of \cite{AMN:2006} or \cite{AMN:2008}). For simple algebras of the complex type, the simple roots are denoted $\{\alpha_1,\ldots,\alpha_\ell,\alpha'_1,\ldots,\alpha'_\ell\}$ with $\epsilon\alpha_j=\alpha'_j$.

\begin{prop}
The simple Levi-Tanaka algebras of kind $2$ are direct sums of  simple factor of the following types:
\begin{enumerate}
\item Type $\mathrm{A_\ell\,III/IV}$, $\Phi=\{\alpha_i\}$, with $1\leq i\leq p$ and $i\neq (\ell+1)/2$ \textup(or $q\leq i\leq\ell$ and $i\neq (\ell+1)/2$\textup)\textup;
\item Type $\mathrm{D_\ell\,Ib/IIIb}$, $\Phi=\{\alpha_\ell\}$ \textup(or $\Phi=\{\alpha_{\ell-1}\}\textup)$\textup;
\item Type $\mathrm{E\,II/III}$, $\Phi=\{\alpha_1\}$ \textup(or $\Phi=\{\alpha_6\}$\textup)\textup;
\item Type $\mathrm A_\ell^\C$, $\Phi=\{\alpha_i,\alpha'_j\}$ with $i\neq j$\textup;
\item Type $\mathrm D_\ell^\C$,  $\Phi=\{\alpha_1,\alpha'_{\ell-1}\}$ or $\Phi=\{\alpha_{\ell-1},\alpha'_{\ell}\}$ \textup(or $\Phi=\{\alpha_1,\alpha'_{\ell}\}$ or $\Phi=\{\alpha_{\ell-1},\alpha'_{1}\}$, $\Phi=\{\alpha_\ell ,\alpha'_{\ell-1}\}$ or $\Phi=\{\alpha_{\ell},\alpha'_1\}$\textup)\textup;
\item Type $\mathrm{E}^\C_6$ with $\Phi=\{\alpha_1,\alpha'_6\}$ \textup(or $\Phi=\{\alpha_6,\alpha'_1\}$\textup).\qed
\end{enumerate}
\end{prop}

\medskip

We fix then a compact quadric $\hat Q$, the corresponding semisimple Levi-Tanaka algebra $\g$, a maximally noncompact Cartan subalgebra $\la h$ of $\g$ contained in $\g_0$, a system $\mathcal B$ of positive simple roots of the root system $\mathcal R=\mathcal R(\g^C,\la h^\C)$.

Of course in this case $E=E_\gs$. 
We recall that $\G^0=\Int(\g)=\Aut_{\CR}(\hat Q)^0$ is the adjoint group, and $\hat Q$ can be identified to the set of $\Ad(\G^0)$-conjugates of $\p$ in $\g$ or of $\q$ in $\g^\C$. We also recall that the \emph{analytic Weyl group} $\gr W(\G^0,\h)$ is the quotient of the normalizer in $\G^0$ of $\h$ by the centralizer  in $\G^0$ of $\h$.

First we prove that in the semisimple case the converse of Lemma~\ref{lem:p-opp-conj} holds true.
\begin{lem}\label{lem:p-opp-conj-bis}
A compact quadric $\hat Q$ has property $(\tilde S)$ if and only if $\p^{\mathrm{opp}}$ is conjugate to $\p$ by an inner automorphism of $\g$.
\end{lem}
\begin{proof}
If $\p^{\mathrm{opp}}$ is conjugate by an inner automorphism to $\p$, we can choose a Weyl group element $w$ with $w\cdot \p=\p^{\mathrm{opp}}$. A representative $\gamma$ of finite order in $\G^0$, which exists thanks to \cite{Tits:1966}, satisfies the $(\tilde S)$ property.
\end{proof}

\subsection{Simple factors of the real type}
We show that, for simple Lie algebras of the real type, the $(\tilde S)$ property always holds true. We recall that the analytic Weyl group of a real connected semisimple Lie group $\G^0$ with respect to a real Cartan subalgebra $\la h$ is the group:
\[ \gr W(\G^0,\la h)= \gr N_{\G^0}(\la h^\C)/\gr Z_{\G^0}(\la h^\C).\]
It is a subgroup  of the usual Weyl group $\gr W(\g^\C,\la h^\C)$.

\begin{lem}\label{lem:w-real}
If $\g$ is a simple algebra of the real types $\mathrm{A\,III/IV}$, $\mathrm{D\,Ib}$, $\mathrm{D\,IIIb}$, $\mathrm{E\,II/III}$ and $\h$ is a maximally split Cartan subalgebra, then the longest element $w_0$ of the Weyl group $\gr{W}(\g^\C,\h^\C)$ is in the analytic Weyl group $\gr{W}(\G^0,\h)$.
\end{lem}
\begin{proof}
First of all, we recall that, for Lie algebras of type $D_\ell$ and $\ell$ even,
the longest element $w_0$ of the Weyl group is minus the identity, 
while in the other cases of the lemma, $w_0$ is equal to
minus the identity composed with the root involution associated to the symmetry of the Dynkin diagram 
(see \cite{Bou68}).

If $\g$ is of type $\mathrm{A\,III/IV}_{\ell}$, then the roots $\beta_j=\e_j-\e_{\ell+2-j}$, $1\leq j\leq  (\ell+1)/2$, are either real or compact, hence the associated symmetries $s_{\beta_j}$ are in the analytic Weyl group. The longest element is $w_0=\Pi_js_{\beta_j}$.

If $\g$ is of type $\mathrm{D\,Ib}_{\ell}$ with $\ell=2k+1$ odd, or of type $\mathrm{D\,IIIb}_{n}$, the roots ${\e_{2i-1}\pm\e_{2i}}$, for $1\leq i\leq k$, are either real or compact, hence the associated symmetries $s_{\e_{2i-1}\pm\e_{2i}}$  are in the analytic Weyl group, and their product is the longest element $w_0$.

If $\g$ is of type $\mathrm{D\,Ib}_{\ell}$ with $\ell=2k$ even, the roots ${\e_{2i-1}\pm\e_{2i}}$, for $1\leq i\leq k$, are real, hence the associated symmetries $s_{\e_{2i-1}\pm\e_{2i}}$  are in the analytic Weyl group, and furthermore also the symmetry $s_{\e_{2k-1}-\e_{2k}}\circ s_{\e_{2k-1}+\e_{2k}}$ belongs to it. Their product is the longest element $w_0$.

If $\g$ is of type $\mathrm{E\,II/III}$,  the roots $\alpha_1+\alpha_3+\alpha_4+\alpha_5+\alpha_6$ and $\alpha_1+2\alpha_2+2\alpha_3+3\alpha_4+2\alpha_5+\alpha_6$ are real, and the roots $\alpha_4$ and $\alpha_3+\alpha_4+\alpha_5$ are either real or compact, hence the associated symmetries are in the analytic Weyl group, and their product is the longest element $w_0$. 
\end{proof}
\begin{prop}\label{prop:realaa}
If $\g$ is a simple algebra of the real types $\mathrm{A\,III/IV}$, $\mathrm{D\,Ib}$, $\mathrm{D\,IIIb}$, $\mathrm{E\,II/III}$,
then there exists an element of finite order $\gamma\in\G^0$ such that $\Ad(\gamma)(E)=-E$.
\end{prop}
\begin{proof}
The longest element $w_0$ of the Weyl group acts on $\h$ either by $-\Id$ or by $-\Id\circ\epsilon$, where $\epsilon$ is the map induced by the nontrivial automorphism of the diagram. Since $E$ is $\epsilon$-invariant, $w_0\cdot E=-E$.

Finally, according to \cite{Tits:1966},  there exists a representative $\gamma$ of $w_0$ in $\G^0$, of order $2$ or $4$.
\end{proof}

\subsection{Simple factors of the complex type}
We consider now the case where $\g$ is a simple algebra of the complex types $A^\C$, $D^\C$, or $E^\C$.
\begin{lem}
If a quadric $\hat Q$ admits an automorphism of finite order $\gamma$ with $\Ad(\gamma)(E)=-E$, then there exists a maximally split Cartan subalgebra, containing $E$, self-conjugate, contained in $\p$, and $\Ad(\gamma)$-invariant.
\end{lem}
\begin{proof}
Let $\Gamma\subset\Aut(\g^\C)$ be the group generated by $\Ad(\gamma)$ and complex conjugation. It is a finite group, and it is the direct product of a cyclic group and $\Z/2\Z$. The subalgebra $\g_0^\C$ is $\Gamma$-invariant.
By \cite{BM:1955} there exists a $\Gamma$-invariant Cartan subalgebra $\h^\C$ of $\g_0^\C$. It contains $E$, because $E$ is in the center. 
Since $\g_0^\C$ contains a Cartan subalgebra of $\g^\C$, also $\h^\C$ is a Cartan subalgebra of $\g^\C$. Finally, $\h=\h^\C\cap\g$ is maximally split because there exists only one conjugacy class of Cartan subalgebras.
\end{proof}
Fix an $S$-adapted Weyl chamber and system of simple positive roots.

In this case the analytic Weyl group $\gr{W}(\G^0,\h)$ is exactly the Weyl group $\gr{W}(\g,\h)$, where $\g$ and $\h$ are considered as complex Lie algebras. Thus the conclusion of Lemma~\ref{lem:w-real} is trivially true.

We recall a result about conjugacy of parabolic subalgebras.

\begin{lem}\label{lem:conj}
Let $\g$ be a complex semisimple Lie algebra, $\la b$ a Borel subalgebra, and $\q,\q'\supset\la b$ two parabolic subalgebras. If $\q'$ is $\Int(\g)$-conjugate to $\q$ then $\q=\q'$.
\end{lem}
\begin{proof}
Assume that $w\in\Int(\g)$ transforms $\q$ into $\q'$.
Let $\la b'=w\cdot \la b$. Both $\la b$ and $ \la b'$ are Borel subalgebras of $\g$ contained in $\q'$. In particular they are Borel subalgebras of $\q'$, hence conjugated by an element $u\in\Int(\q')$, which we can lift to an element $u'\in\Int(\g)$ that preserves $\q'$. Then $u'w\cdot\q=\q'$ and $u'w\cdot\la b=\la b$, and it follows $u'w\cdot\q=\q$.
\end{proof}

\subsection{Conclusion}
From the results above it follows:
 
\begin{thm}\label{thm:semisimple}
A quadric $\hat Q$ with a semisimple associated Levi-Tanaka algebra $\g$ has the $(\tilde S)$ property if and only if the simple factors of $\g$ are all of the following real types:
\begin{enumerate}
\item  $\mathrm{A_\ell\,III/IV}$, $\Phi=\{\alpha_i\}$, with $1\leq i\leq p$ and $i\neq (\ell+1)/2$ \textup(or $q\leq i\leq\ell$ and $i\neq (\ell+1)/2$\textup)\textup;
\item  $\mathrm{D_\ell\,Ib/IIIb}$, $\Phi=\{\alpha_\ell\}$ \textup(or $\Phi=\{\alpha_{\ell-1}\})$\textup;
\item  $\mathrm{E\,II/III}$, $\Phi=\{\alpha_1\}$ \textup(or $\Phi=\{\alpha_6\}$\textup)\textup;
\end{enumerate}
or of the following complex types:
\begin{enumerate}
\item[$(1')$]  $\mathrm A_\ell^\C$ with $\Phi=\{\alpha_j,\alpha'_{\ell+1-j}\}$\textup;
\item[$(2')$]  $\mathrm D_\ell^\C$ with $\ell$ even  and $\Phi=\{\alpha_1,\alpha'_{\ell-1}\}$ or $\Phi=\{\alpha_{\ell-1},\alpha'_{\ell}\}$ \textup(or $\Phi=\{\alpha_1,\alpha'_{\ell}\}$ or $\Phi=\{\alpha_{\ell-1},\alpha'_{1}\}$, $\Phi=\{\alpha_\ell , \alpha'_{\ell-1}\}$ or $\Phi=\{\alpha_{\ell},\alpha'_1\}$\textup)\textup;
\item[$(3')$]  $\mathrm D_\ell^\C$ with $\ell$ odd and $\Phi=\{\alpha_\ell,\alpha'_{\ell-1}\}$ \textup(or $\Phi=\{\alpha_{\ell-1},\alpha'_{\ell}\}$\textup)\textup;
\item[$(4')$]  $\mathrm{E}_6^\C$ with $\Phi=\{\alpha_1,\alpha'_6\}$ \textup(or $\Phi=\{\alpha_6,\alpha'_1\}$\textup).
\end{enumerate} 
\end{thm}
\begin{proof}
For simple factors of the real type the statement is a consequence of Propostion~\ref{prop:realaa}.

For simple factores of the complex type,
the types listed are exactly those for which $\p^{\mathrm{opp}}$ is conjugate to $\p$ for the action of $w_0$. By Lemma~\ref{lem:p-opp-conj-bis} these are the algebras with the $(\tilde S)$ property.

Assume now that  $\p^{\mathrm{opp}}$ is conjugate to $\p$ for the action of some element $w$ of the analytic Weyl group. Then $(ww_0\cdot \p)\cap\p$ contains some Borel subalgebra. From Lemma~\ref{lem:conj} it follows that $ww_0\cdot \p=\p$ that is $\p^{\mathrm{opp}}=w_0\cdot \p$.
\end{proof}
We will see later that for a semisimple $\g$, the $(S)$ property and the $(\tilde S)$ property are equivalent.

\section{The general case}
We drop now the hypothesis that the Levi-Tanaka algebra associated to a quadric $\hat Q$ is semisimple.
\begin{lem}
If a quadric $\hat Q$ has property $(\tilde S)$,
then there exists a $\Ad(\gamma)$-invariant graded Levi factor $\gs$ of $\g$, as described in \S\textup{\ref{s:Levi}}.
\end{lem}
\begin{proof}
Taft \cite{Taft:1964} proves that if $\Gamma$ is a finite group of automorphisms of a real Lie algebra $\g$, and $\la{a}\subset\g$ is a $\Gamma$-invariant semisimple subalgebra, then there exists a $\Gamma$-invariant Levi factor $\la{s}$ and a $\Gamma$-fixed element  $X$ in the nilradical of $\g$ such that $\Ad(\exp X)(\la{a})\subset\la{s}$. Actually his proof is valid for any $\Gamma$-invariant subalgebra $\la{a}$ contained in some (non necessarily invariant) Levi factor. It follows that if $\la{a}$ is a $\Gamma$-invariant subalgebra contained in some Levi factor, then there exists a $\Gamma$-invariant Levi factor $\la s$ containing $\la{a}$.

Let $\Gamma$ be the group generated by $\Ad(\gamma)$, and let $\la a=\C\cdot E$. By Proposition~\ref{Eins}, there exists a Levi factor containing $\la a$. It follows that there exists a $\Gamma$-invariant Levi factor $\gs$ containing $\la{a}$. It is graded, because it contains $E$, and it has a compatible complex structure on $\gs_{-1}$ again by \cite{MN:2002}.
\end{proof}
We fix then a Levi-Malcev decomposition $\g=\gs\oplus\rad$ as in \S\ref{s:Levi}. Since the group $\G^0$ is semi-algebraic, we also have a corresponding Levi decomposition $\G^0=\gr S\gr R$ (note that $\gr S\cap\gr R$ is discrete, but not necessarily trivial).
\begin{lem}
If $\p^\mathrm{opp}$ is $\Int(\g)$-conjugate to $\p$, then $\p^\mathrm{opp}\cap\gs$ is $\Int(\gs)$-conjugate to $\p\cap\gs$.
\end{lem}
\begin{proof}
Decompose an element $\gamma\in\Int(\g)=\G^0$ such that $\Ad(\gamma)(\p)=\p^\mathrm{opp}$ as $\gamma=\gamma_\gr{S}\gamma_\gr{R}$. 
Then $\Ad(\gamma_\gr{S})(\p\cap\gs)=\p^\mathrm{opp}\cap\gs$.
\end{proof}
The simple ideals of the Levi factor $\gs$ belong to three families. Those of kind $2$ are Levi-Tanaka algebras and,
by Lemma \ref{LeviMalcev}, there is at least one of them. 
Those of kind $1$ are of the complex type and correspond to compact hermitian symmetric spaces. We can ignore for the moment those of kind $0$.
\begin{thm}
The quadric $\hat Q$ has property $(\tilde S)$ if and only if $E=E_\gs$ and the simple ideals of kind $2$ of a Levi factor are of the types described in Theorem~\textup{\ref{thm:semisimple}}, and the simple ideals of kind $1$ of a Levi factor are of the following types:
\begin{enumerate}
\item $\mathrm A_\ell^\C$ with $\ell$ odd and $\Phi=\{\alpha_{(\ell+1)/2}\}$;
\item $\mathrm D_\ell^\C$ with $\ell$ even and $\Phi=\{\alpha_1\}$ or $\Phi=\{\alpha_{\ell-1}\}$ or $\Phi=\{\alpha_\ell\}$;
\item $\mathrm D_\ell^\C$ with $\ell$ odd and $\Phi=\{\alpha_1\}$.
\qed
 \end{enumerate} 
 \end{thm}
\begin{proof}
Indeed the same proof as in Theorem~\ref{thm:semisimple} applies to the Levi factor. 
The types listed are exactly those for which $\p^{\mathrm{opp}}$ is conjugate to $\p$ for the action of $w_0$.
The resulting element $\gamma$ is still of finite order in $\G$, because $\gr S$ is a finite covering of $\Int(\gs)$.
\end{proof}

\section{Recovering an involution}
So far we have proved only the existence of a finite order inner automorphism reversing the degree. Now we investigate the existence of an involutive automorphism with this property.

We keep the notation of the previous section. Moreover, let $\hat{\gr S}$ be the universal connected linear group with Lie algebra $\gs$, i.e. the set of real points of the simply connected group with Lie algebra $\gs^\C$. There is a natural projection $\pi\colon\hat{\gr S}\rightarrow\gr S$ which is a finite covering map.

We proceed in two steps. For simple Levi factors of kind $2$, we look for elements $\gamma, \gamma'\in\hat{\gr S}$, with the properties that: (i) $\Ad_{\g}(\gamma)(E)=\Ad_{\g}(\gamma')(E)=-E$,  (ii) $\gamma'^2=e$, (iii) $\gamma^2\in\gr{Z}(\hat{\gr{S}})$, and $\gamma^2|_{V^\lambda}=(-1)^{2\lambda(E)}$ for every irreducible representation $V$ and weight $\lambda$. For simple Levi factors of kind $1$, we provide a general construction for such an element $\gamma$.
In many cases the image of $\gamma^2$ or $\gamma'{}^2$ in $\gr{S}$, and hence in $\gr{G}$, is the identity, and thus we obtain the $(S)$ property. 

We remark that in the following discussion the algebraic structure of the radical does not play any role, and we only consider it as a $\gs$-module.

We introduce the following notation. If $\alpha$ is a root of $\gs$, then let $\gs(\alpha)$ be the (complex) Lie subalgebra isomorphic to $\la{sl}(2,\C)$ containing $\g^\alpha$ and $\g^{-\alpha}$, and $\gr S(\alpha)$ the corresponding analytic subgroup in $\hat{\gr S}$. Let $\tilde s_\alpha$ be the image of $\left(\begin{smallmatrix}0&\ii\\\ii&0\end{smallmatrix}\right)$ in $\gr S(\alpha)$. We have that $\Ad(\tilde s_\alpha)=s_\alpha$, $\tilde s_\alpha^4=1$, and $\tilde s_\alpha^2|_{V^\lambda}=(-1)^{(\alpha^\vee,\lambda)}$ for any representation $V$ and weight $\lambda$. 

\subsection{Simple ideals of kind \texorpdfstring{$2$}{2}}
\subsubsection*{Case $\mathrm{A}_\ell$}
In this case $\lambda(E)\in\Z$ for all weights $\lambda$.
Denote by $A_k$ the $k\times k$ matrix with entries equal to $1$ on the antidiagonal, and $0$ elsewhere.
Let $\gamma\in\gr{S}$ be the block matrix:
\[
\gamma=\begin{pmatrix}&&A_{[\frac{\ell}{2}]}\\&B&\\A_{[\frac{\ell}{2}]}&&\end{pmatrix}
\]
with $B=(1),(-1),I_2, A_2$ depending on the class of $\ell$ modulo $4$, in such a way that $\det\gamma=1$. Then $\gamma$ satisfies our hypotheses.

\subsubsection*{Case $\mathrm{D}_\ell$ with $\ell=2k+1$ odd}
In this case $\lambda(E)\in\Z$ for all weights $\lambda$.
Let $\tilde w_0=\Pi_{i=1}^k\tilde s_{e_{2i-1}+e_{2i}}\tilde s_{e_{2i-1}-e_{2i}}$. Then, if $\{\omega_j\}$ are the fundamental weights,
\[
\tilde w_0^2|_{V^{\omega_j}}=\begin{cases}\Id&\text{if $1\leq j\leq 2k-1$},\\(-1)^{k}\Id&\text{if $j=2k,2k+1$}.\end{cases}
\]

If $k$ is even, i.e. $\ell\equiv 1\pmod{4}$, then $\gamma=\gamma'=\tilde w_0$ satisfies $\gamma^2=1$.

If $k$ is odd, i.e. $\ell\equiv 3\pmod{4}$, then we identify the subalgebra corresponding to $\{\pm\alpha_{\ell-2},\pm\alpha_{\ell-1},\pm\alpha_{\ell}\}$ to $\la{su}(1,3)$ or $\la{su}(2,2)$ or $\la{sl}(4,\C)$ and let $h$ be the image of $\ii\Id$ in the corresponding subgroup. Then $\tilde w_0$ and $h$ commute, and $\gamma=\gamma'=\tilde w_0h$ is the sought after element.

\subsubsection*{Case $\mathrm{D_\ell\,Ib}$ or $\mathrm{D}^\C_\ell$ with $\ell=2k$ even, $\Phi=\{\alpha_{\ell-1}\}$ or $\Phi=\{\alpha_{\ell-1},\alpha'_\ell\}$}
In this case $\omega_j(E)=j\in\mathbb Z$ for the fundamental weights $\omega_1,\dots,\omega_{\ell-2}$, and $\omega_{\ell-1}(E)=\omega_\ell(E)=(\ell-1)/2$.
Let $\tilde w_0=\Pi_{i=1}^k\tilde s_{e_{2i-1}+e_{2i}}\tilde s_{e_{2i-1}-e_{2i}}$. Then, if $\{\omega_j\}$ are the fundamental weights,
\[
\tilde w_0^2|_{V^{\omega_j}}=\begin{cases}\Id&\text{if $1\leq j\leq 2k-2$},\\(-1)^{k}\Id&\text{if $j=2k-1,2k$}.\end{cases}
\]

If $k$ is odd, i.e. $\ell\equiv 2\pmod{4}$, then $\gamma=\tilde w_0$ satisfies $\gamma^2|_{V^\lambda}=(-1)^{2\lambda(E)}$.  Let $I\in\gr{Spin}(\ell-1,\ell+1)$ be an element covering  $-\Id\in\gr{SO}(\ell-1,\ell+1)$. Then $\gamma'=(I\cdot\tilde w_0)$ satisfies $\gamma'{}^2=\Id$

If $k$ is even, i.e. $\ell\equiv 0\pmod{4}$, then $\gamma'=\tilde w_0$ satisfies $\gamma'{}^2=\Id$. In general however it is not possible to find an element $\gamma$ with the required properties.

\subsubsection*{Case $\mathrm{D}_\ell^\C$ with $\ell=2k$ even, $\Phi=\{\alpha_1,\alpha_{\ell-1}'\}$}
In this case $\omega_i(E)\in\frac12\Z$ for all fundamental weights $\omega_i$, and $\omega_i(E)\in\Z$ exactly for $\omega_2,\omega_4,\dots,\omega_{2k-2}$ and for $\omega_{\ell-1}$ (resp. $\omega_\ell$) if $\ell\equiv 0 \pmod{4}$ (resp. if $\ell\equiv 2\pmod{4}$).

As in the previous case, there exists an element $\gamma'$ with $\gamma'{}^2=e$ satisfying all conditions. In general however it is not possible to find an element $\gamma$ with the required properties.

\subsubsection*{Case $\mathrm{E}_6$}
In this case $\lambda(E)\in\Z$ for all weights $\lambda$.
Let 
\[
\gamma=\gamma'=\tilde w_0=\tilde s_{\alpha_1+\alpha_3+\alpha_4+\alpha_5+\alpha_6}\tilde s_{\alpha_1+2\alpha_2+2\alpha_3+3\alpha_4+2\alpha_5+\alpha_6}\tilde s_{\alpha_4}\tilde s_{\alpha_3+\alpha_4+\alpha_5}.
\]
Then $\gamma^2=e$

\medskip

Summarizing, we found an element $\gamma'$ for all simple factors of kind $2$, and an element $\gamma$ for  all simple factors of kind $2$ excepts some those of kind $\mathrm D_\ell$ with $\ell$ even.

\subsection{Simple ideals of kind \texorpdfstring{$1$}{1}}
First we consider the existence of a suitable element $\gamma$.
The longest element $w_0$ of the Weyl group can be written as a product of reflections
\[
w_0=\Pi s_{\beta_i}
\]
where $\{\beta_i\}$ is a maximal set of positive strongly orthogonal roots. Let $\{\alpha_j\}\subset\{\beta_i\}$ be the subset of roots of degree $1$ (i.e. $\alpha_j(E)=1$), and $w_1=\Pi s_{\alpha_j}$, $\gamma=\Pi \tilde s_{\alpha_j}$.

Since $w_1(E)=-E$, we have 
\[
E=\sum_j\frac{\alpha_j(E)\alpha_j}{(\alpha_j,\alpha_j)}
=\frac12\sum_j\alpha_j(E)\alpha_j^\vee=\frac12\sum_j\alpha_j^\vee.
\]
Then 
\[
\gamma|_{V^\lambda}=\Pi_j(-1)^{(\alpha_j^\vee,\lambda)}=(-1)^{(\sum_j\alpha_j^\vee,\lambda)}=(-1)^{2\lambda(E)}.
\]

We turn now to the problem of the existence of $\gamma'$ with $\gamma'{}^2=1$. 
For simple ideals of type $\mathrm D^\C_\ell$ or $\mathrm E^\C_6$ the element $\gamma'$ found in the previous subsection is a representative of the longest element of the Weyl group, thus satisfies all requirements. 
For simple ideals of type $\mathrm A^\C_\ell$ with $\ell\equiv 3\pmod 4$, the matrix with entries equal to $1$ on the antidiagonal and $0$ elsewhere provides the element $\gamma'$.
For simple ideals of type $\mathrm A^\C_\ell$ with $\ell\equiv 1\pmod 4$ there is no such element $\gamma'$.

\begin{prop}\label{prop:ie}
Let $\hat Q$ be a quadric with the $(\tilde S)$ property, and $\g$ the associated Levi-Tanaka algebra, with Levi-Malcev decomposition $\g=\gs\oplus\rad$. If any of the following conditions is satisfied, then $\hat Q$ has the $(S)$ property:
\begin{enumerate}
\item $\g$ is semisimple;
\item $\gs$ does not contain any simple factor of kind $1$ and type $\mathrm A^\C_\ell$ with $\ell\equiv 1\pmod 4$,
\item $\gs$ does not contain any simple factor of kind $2$ and type $\mathrm D_\ell$ with $\ell$ even.
\end{enumerate}
\end{prop}
\begin{proof}
In case (1) all the ideals of $\gs$ are of kind $2$, so case (1) is a subcase of case (2).

For each simple ideal $\gs_i$ of $\gs$, let $\gamma_i,\gamma'_i$ be the images in $\gr S_i$ of the elements described in the previous sections, if defined.
For Levi factors of kind $0$ we let $\gamma_i=\gamma'_i=e$.

In case (2) the elements $\gamma_i$ are defined for every simple factor $\gs_i$, and we let $\gamma=\Pi_i \gamma_i$.
In case (3) the elements $\gamma_i'$ are defined for every simple factor $\gs_i$, and we let $\gamma=\Pi_i \gamma_i'$.
In both cases the element $\gamma\in\G$ has order $2$.
\end{proof}

Since compact quadrics have a semisimple group of automorphisms, we have the following.
\begin{cor}
Every compact CR quadric has propoerty $(S)$.\qed
\end{cor}
\begin{remark}
If a quadric has property $(\tilde S)$ the above construction shows that it is anyway possible to find an appropriate automorphism $\gamma$ with order $2$ or $4$.
\end{remark}
\begin{remark}
As the next example shows, the conditions in Proposition~\ref{prop:ie} are not necessary. In fact we have no example of quadrics with the $(\tilde S)$ property but without the $(S)$ property.
\end{remark}
\begin{exam}
Let $\gs=\la o(8,\C)\oplus\la{sl}(2,\C)$, with the grading and the CR structure defined respectively by:
\begin{align*}
E&=\mathrm{diag}(1,1,1,0,0,-1,-1,-1) \oplus \mathrm{diag}(1/2,-1/2),\\
J&=\mathrm{diag}(0,0,0,\ii,-\ii,0,0,0) \oplus \mathrm{diag}(\ii/2,-\ii/2).
\end{align*}
Let $\C^8$ and $\C^2$ denote the standard representations of $\la o(8,\C)$ and $\la{sl}(2,\C)$, respectively, and let $V=\C^8\otimes\C^2$. The same elements $E$ and $J$ define a grading and CR structure on the semidirect product $\gs\oplus V$. We finally define $\g=\gs\oplus V\oplus\C T$, where $T$ is an element commuting with $\gs$ and such that $\mathrm{ad}(T)|_V=\Id$. Then $\g$ is a Levi-Tanaka algebra associated to a quadric with the $(\tilde S)$ property. It has the $(S)$ property too, with the element $\gamma=\gamma'_{\la o(8,\C)}\gamma_{\la{sl}(2,\C)}\exp(\ii\pi T/2)$, but there is no such element $\gamma$ in $\gr S$.
\end{exam}

\section{An example}
The CR dimension of a quadric $\hat Q$ is $n=\dim_{\R}\g_{-1}/2$, while the CR-codimension is $k=\dim_{\R}\g_{-2}$, and hence  $\dim_{\R}\g_{-}=2n+k$. For quadrics with the $(S)$ or $(\tilde S)$ property, the dimension of $\g_+=\Ad(\gamma)(\g_-)$ is $2n+k$.
It was an open question whether also in the general case the dimension of $\g_+$ can be estimated by $2n+k$
(see, for example, \cite[p.445]{EzSc:2001}). A first negative answer was given by P.B. Utkin in 2002 (see \cite{Ut:2002}).
Here we give a new example.
\begin{exam}
For $n=7$ and $k=8$, we consider the quadric $\hat Q=\hat Q^H$ associated to the hermitian form $H$ parametrized by
$\alpha,\beta,\gamma,\delta\in\C\simeq\R^2$:
$$
\begin{pmatrix}
0&\bar\alpha&0&0&\bar\gamma&0&\bar\delta\\
\alpha&0&\beta&\gamma&0&\delta&0\\
0&\bar\beta&0&0&0&0&0\\
0&\bar\gamma&0&0&0&0&0\\
\gamma&0&0&0&0&0&0\\
0&\bar\delta&0&0&0&0&0\\
\delta&0&0&0&0&0&0
\end{pmatrix}.
$$
Let $\gs={\la{sl}}(3,\C)=\bigoplus_{i=-2}^{2}\gs_i$ be endowed with the unique Levi-Tanaka structure, given by the elements
\begin{equation*}
E^\gs=\mathrm{diag}(1,0,-1),\qquad J^\gs=\mathrm{diag}(-\ii/3,2\ii/3,-\ii/3).
\end{equation*}
Let $V=\C^3$ be the space of the standard representation $\rho$ of $\gs$ and 
$U^1,U^2$ two copies the adjoint representation of $\gs$. 
We assume 
on $V$ the grading $V=V_{-2}+V_{-1}+V_0$ given by the eigenspace decomposition of $(\rho(E^\gs)-\Id)$
and on $V_{-1}$ the complex structure given by multiplication by the imaginary unit $J=\ii\Id$.

On $U^k, k=1,2$, we put the grading $U^k=\oplus_{i=-2}^{2}U^k_i$ induced by $\Ad(E^\gs)$ and the complex structure induced by $\Ad(J^\gs)$.

On $\la h=\gs\oplus V\oplus U^1\oplus U^2$ we have a natural Lie algebra structure, with $V\oplus U^1\oplus U^2$ an abelian ideal and $\gs$ acting through the standard or adjoint representation. Then $\h$ is a graded Lie algebra, with a complex structure on $\h_{-1}$ and it is fundamental, nondegenerate and transitive.

Let $W^1,W^2$ be two copies of the dual space $V^*$ of $V$. The algebra $\gs$ acts on them via the contragradient representation $-{}^t\!\rho$.
We assume on $W^k, k=1,2$ a grading $W^k=W^k_0+W^k_1+W^k_2$ given by the eigenspace decomposition of $(-{}^t\!\rho(E^\gs)+\Id)$.
We define a product of elements of $V$ and $W^k$ 
$$[v,w]:=v\otimes w$$ 
with values in $V\otimes W^k $, which we identify with $U^k\oplus\C \simeq {\la{gl}}(n,\C)$. 

Assuming $W^1+W^2+U^1+U^2+\C^2$ abelian, we obtain a graded Lie algebra 
$\ga=\gs+V+W^1+W^2+U^1+U^2+\C^2$ which is nondegenerate and fundamental.
It is also pseudocomplex and transitive (see \cite{MN:1997}).
Its maximal pseudocomplex prolongation $\g=\oplus_{-2}^{2}\g_i$ is finite dimensional with
$\dim \g_1\geq \dim \ga_1=16>14=\dim \g_{-1}$ and
$\dim \g_2\geq \dim \ga_2=10>8=\dim \g_{-2}$.
\end{exam}


\providecommand{\bysame}{\leavevmode\hbox to3em{\hrulefill}\thinspace}
\providecommand{\MR}{\relax\ifhmode\unskip\space\fi MR }
\providecommand{\MRhref}[2]{%
  \href{http://www.ams.org/mathscinet-getitem?mr=#1}{#2}
}
\providecommand{\href}[2]{#2}

\end{document}